\documentclass[a4paper,12pt]{article}
\usepackage{amsmath,amsthm,amssymb,amsfonts}
\usepackage{bbm,bm}
\usepackage{latexsym}
\usepackage{mathrsfs}
\usepackage{threeparttable}
\usepackage{tabularx}
\usepackage{booktabs}
\usepackage{graphicx,subfigure}
\usepackage{color}
\usepackage{indentfirst}
\usepackage{geometry}
\usepackage{caption}
\usepackage{lineno}
\usepackage{abstract}
\usepackage{enumerate,mdwlist}
\usepackage[numbers,sort&compress]{natbib}
\usepackage[colorlinks,
            linkcolor=blue, 
            citecolor=red  
            ]{hyperref}
\usepackage{epstopdf}

\makeatletter
\@addtoreset{equation}{section}

\newtheorem{thm}{Theorem}[section]

\newtheorem{lem}[thm]{Lemma}
\newtheorem{cor}[thm]{Corollary}

\newtheorem{pb}[thm]{Problem}

\newtheorem{remark}[thm]{Remark}

\begin{document}

\setlength{\baselineskip}{20pt}
\underline{}\begin{center}{\Large \bf Components of domino tilings under flips in quadriculated cylinder and torus}\\
\vspace{4mm}
{Qianqian Liu, Jingfeng Wang, Chunmei Li, Heping Zhang\footnote{The corresponding author.\\ \indent\indent\textsf{E-mails: liuqq2016@lzu.edu.cn, wangjingfeng2@huawei.con, 283756111@qq.con, zhanghp@lzu.edu.cn.}}}
\vspace{2mm}

{ School of Mathematics and Statistics, Lanzhou University, Lanzhou, Gansu 730000,\\ P. R. China}\\
\end{center}
\begin{abstract} In a region $R$ consisting of unit squares, a domino is the union of two adjacent squares and a (domino) tiling is a collection of dominoes with disjoint interior whose union is the region. The flip graph $\mathcal{T}(R)$ is defined on the set of all tilings of $R$ such that two tilings are adjacent  if we change one to another by a flip (a $90^{\circ}$ rotation of a pair of side-by-side dominoes). It is well-known that  $\mathcal{T}(R)$ is connected when $R$ is simply connected. By using graph theoretical approach, we show that the flip graph of $2m\times(2n+1)$ quadriculated cylinder  is still connected, but  the flip graph of $2m\times(2n+1)$ quadriculated torus is disconnected and consists of exactly two isomorphic components.
For a tiling $t$, we associate an integer $f(t)$, forcing number,  as the minimum number of dominoes in $t$ that is contained in no other tilings. As an application, we obtain that the forcing numbers of all tilings in $2m\times (2n+1)$ quadriculated cylinder and torus form respectively an integer interval whose maximum value is $(n+1)m$.\\

\end{abstract}



 {\setcounter{section}{0}
\section{\normalsize Introduction}\setcounter{equation}{0}
In a quadriculated region consisting of unit squares with vertices in $\mathbb{Z}^2$, a \emph{domino} or a \emph{dimer} is the union of two adjacent squares, and a \emph{tiling} is a way to placing dominoes on the region so that there are no gaps or overlaps.
Enumeration of tilings,  known as the \emph{dimer problem}, has applications in statistical mechanics \cite{T2}. In 1961, Kasteleyn \cite{Ka} gave a formula for the number of domino tilings in a quadratic lattice, while Temperley and Fisher \cite{TF} arrived at the same result using a different method. John et al. \cite{J} gave a determinant of smaller order to count  tilings of a polyomino (simply-connected quadriculated region). Moreover, Elkies et al. \cite{E} obtained that Aztec diamond of order $n$ has exactly $2^{\frac{n(n+1)}{2}}$ domino tilings. 
Kim et al. \cite{K} counted the number of domino tilings for three variants of augmented Aztec diamond by Delannoy paths.

In the past  thirty years, some profound structure properties on the set of all tilings  of a fixed region has been established. For a tiling,  a local \emph{flip} is to remove  two parallel dominoes and place them back in the only possible different position. We define the \emph{flip graph} $\mathcal{T}(R)$ of a region $R$ on the set of all tilings whose vertices are all tilings and two tilings are adjacent if we obtain one from the other by a flip.
For a simply connected region,  Thurston \cite{Th90} obtained that the  flip graph  is connected, and Zhang \cite{Z4} showed that it is maximally connected (i.e its connectivity equals the minimum degree). For  a quadriculated region  with holes, using the homology theory Saldanha et al. \cite{T1}  gave three characterizations (combinational, homology and height section versions) for two tilings in the same connected component of the flip graph, and a distance formula between two tilings as well. The homology version of the above characterizations was also extended to quadriculated torus and Klein bottle (only surfaces of positive genus). Also for a simply connected region, R\'{e}mila \cite{RE} studied how each tiling  can be encoded by a height function and obtained that such an encoding induces a distributive lattice structure on the set of the tilings.

Domino tiling of a quadriculated region in plane has been extended in two ways. One way is  to consider a three-dimensional cubiculated manifold, and a domino as a pair of adjacent unit cubes.  Freire et al. \cite{LW17}  characterized  connected components of  tiling space of three-dimensional cubiculated region and torus under flips using  topological invariants twist and flux.
Saldanha \cite{S} obtained normal distribution of the twist of tilings in some three-dimensional cylinders.

The other way is to consider tilings of a  region of other regular polygons, for example, tilings of a triangulated region with calissons (i.e. lozenges whose sides have unit length, formed from two equilateral triangles sharing an edge)(see \cite{O, Th90}).

In the view of graphs, a quadriculated (resp. triangulated) region $D$ in $\mathbb{R}^2$ is a plane graph.
The \emph{inner dual} $D^{\#}$  is also a plane graph: the center of each square (resp. triangle) is placed a vertex,  and  two vertices are connected by an edge $e^*$ if the corresponding squares (resp. triangles) have an edge $e$ in common ($e^*$ only pass through edge $e$ of $D$). A domino (resp. lozenge) in $D$ corresponds to an edge in $D^{\#}$ and a tiling  in $D$ to a perfect matching in $D^{\#}$; for example, see  Fig. \ref{23}. Similar definitions and correspondences can be made for quadriculated surfaces or cubiculated manifolds.
\begin{figure}[h]
\centering
\includegraphics[height=2.8cm,width=14.5cm]{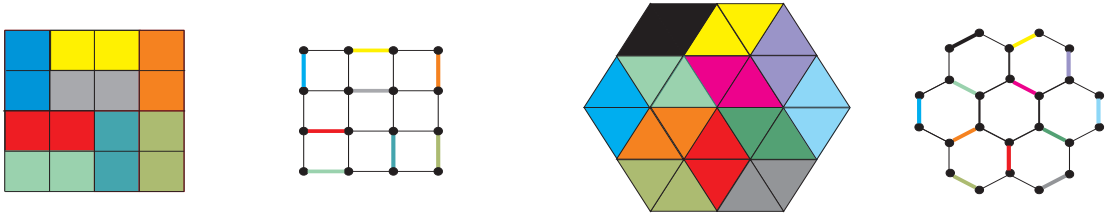}
\caption{\label{23}A domino (resp. lozenge) tiling and the corresponding perfect matching.}
\end{figure}

In particular, if the interior of $D$ is simply-connected, then $D^{\#}$ is a connected subgraph of polyomino graph (resp. hexagonal graph or system) such that each interior face is a square (resp. hexagon).
 A hexagonal system with a perfect matching can be viewed as the carbon-skeleton of a benzenoid hydrocarbon. The flip graph of a quadriculated region $D$ is equivalent to the
 resonance graph of $D^{\#}$ on the set of perfect matchings. The resonance graph was first introduced for a hexagonal system in \cite{Z1,G}, then extended to a plane (bipartite) graph \cite{Z3}. A suitable acyclic orientation for the resonance graph of a plane elementary bipartite graph implies the distributed lattice structure  and median property on the set of perfect matchings; for details, see \cite{Pr93, LZ, Zh06, ZLS08} .

In this article we consider quadriculated cylinder and torus.  Given an $m\times n$  chessboard ($m$ rows, each consists of $n$ squares) for  integers $m\ge 2$ and $n\ge 2$. Roll this rectangle into a cylinder so that the left and right sides stick together and results in $(m,n)$-quadriculated cylinder, denoted by $C(m,n)$. Further identifying top and bottom sides results in $(m,n)$-quadriculated torus, denoted by $T(m,n)$. Then the inner duals of $C(m,n)$ and $T(m,n)$ correspond to $C(m-1,n)$ and $T(m,n)$ respectively.

By using the resonance graph, we obtain that flip graph of quadriculated cylinder $C(2m, 2n+1)$ is connected, but the flip graph of quadriculated torus $T(2m, 2n+1)$ is disconnected, and consists of exactly two isomorphic components, which are equivalent under an automorphism of $T(2m, 2n+1)$.

Motivated by forcing a perfect matching,   we associate a tiling $t$ an integer $f(t)$, called {\em forcing number},  as the minimum number of dominoes in $t$ that is contained in no other tilings. Since a flip of a tiling of a quadriculated region does not change the forcing number by more than one, as an application of connectedness of the flip graph, we obtain that the forcing numbers of all tilings in both $2m\times (2n+1)$ quadriculated cylinder and torus form respectively an integer interval,  whose maximum value   $(n+1)m$ is  proved specially.

\section{\normalsize Preliminary and resonance graphs}
We have seen that tilings of a region can be transformed in perfect matchings of graphs. So we shall use graph-theoretical method to get our main results. First recall some useful concepts of graphs and embeddings.  Let $G$ be a graph with vertex set $V(G)$ and edge set $E(G)$. An edge subset $M$ of $G$ is called \emph{a perfect matching} if each vertex of $G$ is incident with exactly one edge in $M$. For a perfect matching $M$ of $G$, a cycle is \emph{M-alternating} if its edges appear alternately in $M$ and off $M$. Given an $M$-alternating cycle $C$ of $G$, the symmetric difference $M\oplus  E(C)$ is another perfect matching of $G$, and such an operation is called a \emph{flip} or \emph{rotation} of $M$ (along $C$), where the symmetric difference is defined as $A\oplus B:=(A\setminus B)\cup(B\setminus A)$ for sets $A$ and $B$.

Let $G$ be a graph cellularlly embedded on a (closed) surface, i.e.  each face is homeomorphic to an open disc in $\mathbb{R}^2$. Let $F(G)$ denote the set of all faces of $G$.  A face $f\in F(G)$ is \emph{even} (resp. \emph{odd}) if it is bounded by an even (resp. odd) cycle. For $f\in F(G)$, the boundary of  $f$ is denoted by $\partial f$. If $\partial f$ is a cycle, then we call it a \emph{facial cycle}. For convenience, sometimes  we do not distinguish  a face with its boundary, and a cycle with its edge set.

For a given  $F\subseteq F(G)$, the \emph{resonance graph} (or \emph{$Z$-transformation graph}) of $G$ with respect to $F$, denoted by $R(G; F)$, is a graph, where the vertices are   the perfect matchings of $G$ and two perfect matchings $M_1$ and $M_2$ are adjacent if and only if $M_1\oplus M_2$ consists of exactly one facial cycle in $F$, that is, $M_2$ is obtained from $M_1$ by a rotation along a facial cycle in $F$.
In particular,  $R(G; F(G))$ is written $R_t(G)$ which is called the \emph{total resonance graph} of $G$. For a plane graph $G$, there is exactly one unbounded face  which is called the \emph{outer face} of $G$, and others are called \emph{interior faces}.
If $F$ is the set of all interior faces of $G$, then $R(G; F)$ is simply written $R(G)$.

An edge $e$ of $G$ is \emph{allowed} if it lies in some perfect matching of $G$ and \emph{forbidden} otherwise.
A graph is said to be \emph{elementary} if its allowed edges form a connected subgraph.

\begin{lem}\cite{14}\label{2} A bipartite graph $G$ is elementary if and only if $G$ is connected and each edge is allowed.
\end{lem}

Suppose that $X$ is a nonempty subset of $V(G)$. The subgraph of $G$ whose vertex set is $X$ and whose edge set is the set of those edges of $G$ that have both end vertices in $X$ is called the subgraph of $G$ \emph{induced} by $X$ and is denoted by $G[X]$. The induced subgraph $G[V(G)\setminus X]$ is denoted by $G-X$.
If $G'$ is a subgraph of $G$, then $G'$ is called \emph{nice} if $G-V(G')$ has a perfect matching. So an even cycle $C$ of a graph $G$ with a perfect matching is nice if and only if $G$ has a perfect matching $M$ such that $C$ is an $M$-alternating cycle.

Let $C$ be a cycle of a plane graph $G$. By the well-known Jordan Curve Theorem, $C$ divides the plane into exactly two arcwise connected components (interior and exterior of $C$), which both have $C$ as the boundary. Let $I[C]$ and $O[C]$ denote the subgraphs of $G$ consisting of $C$ together with its interior and exterior respectively.

 Zhang et al. obtained some results for the resonance graph as follows.
\begin{thm}\cite{Z3}\label{1.3.3} Let $G$ be a plane elementary bipartite graph. Then $R(G)$ is connected.
\end{thm}

\begin{thm}\cite{Z6}\label{1.3.4} Let $G$ be a plane bipartite graph with a perfect matching. Then $R(G)$ is connected if and only if $I[C]$ is elementary for each nice cycle $C$ of $G$.
\end{thm}

\begin{thm}\cite{Z6} Let $G$ be a plane bipartite graph with a perfect matching. Then $R_t(G)$ is connected if and only if $I[C]$ or $O[C]$ is elementary for each nice cycle $C$ of $G$.
\end{thm}

Let $G$ be a graph cellularly embedded on a surface. Then the \emph{dual graph} $G^*$ of $G$ is defined as follows. $G^*$ is a multigraph that has exactly one vertex in each face of $G$. If $e$ is an edge of $G$, then $G^*$ has an edge $e^*$ crossing $e$ exactly once ($e^*$ has no other points in common with $G$) and joining the two vertices of $G^*$ in the two faces of $G$ having edge $e$ in common (if $e$ is a cut edge of $G$, then the two faces are the same and $e^*$ is a loop). We can see that $G^*$ is always connected.

A graph is \emph{bipartite} if its vertex set can be partitioned into two subsets $X$ and $Y$ so that every edge has one end vertex in $X$ and the other in $Y$.
Tratnik and Ye \cite{TY} obtained the following result.
\begin{thm}\cite{TY} \label{ty} Let $G$ be a graph embedded in a surface and $F\subset F(G)$.
Then $R(G;F)$ is bipartite.
\end{thm}

However, $R_t(G)$ is not necessarily bipartite. For example,  the total resonance graph of a domino is a triangle. However we have the following general result.
\begin{thm}\label{bipar}Let $G$ be a connected graph cellularly embedded on an orientable surface $S$ of genus $h~(h\geq 0)$. If $G$ has  an even number of edges and a perfect matching, then $R_t(G)$ is bipartite.
\end{thm}
\begin{proof}If $R_t(G)$ contains no cycles, then we have done. So we may assume that $R_t(G)$ has a cycle $M_1M_2\cdots M_kM_{k+1}(=M_1)$ so that  $M_{i+1}=M_i\oplus S_i$, where  $S_i$ is a facial cycle of $G$, for $i=1,2,\dots,k$. It suffices to prove that $k$ is even.

For a face $f$ of $G$, we denote by $\delta(f)$ the number of times $\partial f$ appeared in the sequence $S_1,S_2,\dots,S_k$. In particular, if $\partial f$ is not a cycle, then we have $\delta(f)=0$.
For an edge $e$ of $G$, let $f_1$ and $f_2$ denote the two faces of $G$ sharing $e$. We claim that $\delta(f_1)+\delta(f_2)\equiv 0\text{ (mod 2)}$.
Since $M_1M_2\cdots M_kM_1$ is a cycle of $R_t(G)$, we have $$M_1=M_k\oplus S_k=M_1\oplus S_1\oplus S_2\oplus\cdots \oplus S_k, $$
which yields $S_1\oplus S_2\oplus\cdots \oplus S_k=\emptyset$. Hence there are exactly even numbers of terms in the sequence $S_1,S_2,\dots,S_k$ that contains $e$. That is,   $\delta(f_1)+\delta(f_2)\equiv 0 \text{ (mod 2)}$.

Since the dual graph $G^*$ of $G$ is connected, all $\delta(f)$, $f\in F(G)$, have the same parity.  If $\delta(f)\equiv 0 \text{ (mod 2)}$ for $f\in F(G)$, then $k=\sum_{f\in F(G)}\delta(f)\equiv 0\text{ (mod 2)}$. Next we may assume that $\delta(f)\equiv 1 \text{ (mod 2)}$ for $f\in F(G)$. Then $k=\sum_{f\in F(G)}\delta(f)\equiv\sum_{f\in F(G)} 1 \equiv|F(G)| \text{ (mod 2)}.$ Since both $|V(G)|$ and $|E(G)|$ are even, Euler's Formula, $|V(G)|-|E(G)|+|F(G)|=2-2h$, implies that  $|F(G)|$ is also even. So $k$ is even.
\end{proof}

 $C(m,n)$ can be viewed as a graph in the sphere by adding a cap in each open end, which has exactly two faces of size $n$ and others of size 4 for $n\neq 4$.
 By Theorems \ref{ty} and \ref{bipar}, we have the following result.
\begin{cor}Let $n\geq 2$ and $m\geq 2$ be integers with even $mn$. Then $R_t(C(m-1,n))$ and $R_t(T(m,n))$ are bipartite. Hence $\mathcal{T}(C(m,n))$ and $\mathcal{T}(T(m,n))$ are bipartite.
\end{cor}
\begin{proof}Since $mn$ is even, $C(m-1,n)$ and $T(m,n)$ have a perfect matching. Since each vertex of $T(m,n)$ is of degree 4, $T(m,n)$ has $2mn$ edges by degree sum formula. Since $T(m,n)$ is a toroidal graph, by Theorem \ref{bipar}, $R_t(T(m,n))$ is bipartite.

Since $C(m-1,n)$ has $2n$ vertices of degree 3, other vertices are of degree 4, $C(m-1,n)$ has $(2m-1)n$ edges by degree sum formula. If $n$ is even, then $C(m-1,n)$ has even edges.  By Theorem \ref{bipar}, $R_t(C(m-1,n))$ is bipartite. If $n$ is odd, then $C(m-1,n)$ has two odd faces. Since an odd face cannot be the symmetric difference of some two perfect matchings, by Theorem \ref{ty}, $R_t(C(m-1,n))$ is bipartite.
\end{proof}

\section{\normalsize The flip graph of quadriculated cylinder}
For a plane bipartite graph $G$ with a perfect matching, Theorem \ref{1.3.4} says that $R(G)$ is connected if and only if $I[C]$ is elementary for each nice cycle $C$ of $G$. In this section, we will extend the result to plane graphs. As a corollary we obtain that the flip graph of $C(2m,2n+1)$ is connected.
First we give a  result for a plane bipartite graph.
\begin{lem}\cite{Z5}\label{2.1.1} A plane graph is bipartite if and only if each face is even.
\end{lem}

\begin{thm}\label{2.2.2} Let $G$ be a connected plane graph with a perfect matching. Then $R(G)$ is connected if and only if $I[C]$ is an elementary bipartite graph for each nice cycle $C$ of $G$.
\end{thm}
\begin{proof}Sufficiency. Let $M_1$ and $M_2$ be any two perfect matchings of $G$. We will prove that there is a path in $R(G)$ between $M_1$ and $M_2$. We apply induction on the number $m$ of cycles in $M_1\oplus M_2$.  For $m=1$, let $C=M_1\oplus M_2$. Since  $C$ is a nice cycle of $G$, $I[C]$ is an elementary bipartite graph. Then the restrictions $M_i'=M_i|_{I[C]}$,  $i=1,2$,  are two perfect matchings of $I[C]$. By Theorem \ref{1.3.3}, $R(I[C])$ has a path between them: $M^1(=M_1')M^2\cdots  M^{l}(=M'_2).$
Setting $M=M_1\setminus M_1'$,  we have that
$$P'':=(M^1\cup M)(M^2\cup M)\cdots (M^{l}\cup M)$$ is a path in $R(G)$ between $M_1$ and $M_2$.
For $m\geq 2$, take  a cycle $C$ in $M_1\oplus M_2$.  Then $M_3=M_2\oplus C$ is a perfect matching of $G$. Moreover, $M_1\oplus M_3$ contains exactly $m-1$ cycles and $M_2\oplus M_3=C$. By the induction hypothesis, $R(G)$ has a path between $M_1$ and $M_3$ and a path between $M_2$ and $M_3$, and thus has  a path between $M_1$ and $M_2$.

Necessity. Let $C$ be any nice cycle of $G$. We will show that  $I[C]$ is an elementary bipartite graph.
Since $C$ is a nice cycle, $G$ has two perfect matchings $M_1$ and $M_2$ such that $M_1\oplus M_2=C$. Since $R(G)$ is connected, it has a path $P=M^0(=M_1)M^1M^2\cdots M^t(=M_2)$ between $M_1$ and $M_2$.  Then $S_i:=M^{i-1}\oplus M^{i}$ is the boundary of an interior face for $1\leq i\leq t$, and
$$C=M_1\oplus M_2=M^0\oplus M^t=S_1\oplus S_2\oplus \cdots \oplus S_t.$$ For a face $f$, let $\delta(f)$ be the times  $\partial f$ appearing  in the sequence $S_1,S_2, \dots, S_t.$
For two faces $f$ and $f'$ that have an edge $e$ in common, we have
 \begin{equation*}
 \delta(f)+\delta(f')=
 \begin{cases}
 0 \text{ (mod 2)}, & \quad {\text {if $e\notin C$}};\\
 1 \text{ (mod 2)}, &\quad {\text {if $e\in C$}}.
 \end{cases}
 \end{equation*}

We claim that the subgraphs of the dual graph $G^*$ induced by faces lying in the interior and exterior of $C$ respectively are connected. Let $x^*$ and $y^*$ be any two vertices of $G^*$ lying in the interior (resp. exterior) of $C$. By Jordan Curve Theorem, the interior (resp. exterior) of $C$ is arcwise connected. There is a curve connecting $x^*$ and $y^*$ which avoids all vertices of $G$. Then the sequence of faces and edges of $G$ traversed by this curve corresponds to a walk in $G^*$ connecting $x^*$ and $y^*$. So there is a path between $x^*$ and $y^*$ and the claim holds.
For the outer face $f_0$ of $G$, we have $\delta(f_0)=0$.
By the claim, $\delta(f)$ is even or odd according to a face $f$ of $G$ lies in the exterior or interior of $C$. Hence $\partial f$ is a nice cycle of $G$ for each face $f$ in the interior of $C$. So, by Lemma \ref{2.1.1} $I[C]$ is a connected bipartite graph and each edge is allowed in $G$. Next we show that each edge $e$ of $I[C]$ is allowed in $I[C]$. If $e\in M_1\cup M_2$, then it is trivial since the restriction of $M_1$ and $M_2$ on $I[C]$ are perfect matchings. Otherwise, $G$ has a perfect matching $M$ that contains $e$.  Let $C'$ be a cycle of  $ M\oplus M_1$ including $e$.  If $C'\subseteq I[C]$, then $M_1|_{I[C]}\oplus C'$ is a perfect matching of $I[C]$ that contains $e$, and $e$ is allowed in $I[C]$. Otherwise, take a path $P$ of $C'$ in the interior of $C$ such that only the end vertices $u$ and $v$ of $P$ lie in $C$. Obviously $P$ is an $M$ and $M_1$-alternating path whose two edges incident to $u$ and $v$ belong to $M$, and thus $P$ is of odd length. Since $I[C]$ is bipartite, two paths of $C$ between $u$ and $v$ are also of odd length, and one and $P$ form an $M_1$-alternating cycle in $I[C]$. As the above, $e$ is also an allowed edge of  $I[C]$. By Lemma \ref{2}, $I[C]$ is elementary. Summing up the above, we have that $I[C]$ is an elementary bipartite graph.
\end{proof}

Zhang \cite{Z4} gave a criterion to decide if a polyomino is elementary or not.
\begin{lem}\cite{Z4}\label{2.1.3} Let $G$ be a polyomino. Then $G$ is elementary if and only if the boundary of $G$ is a nice cycle.
\end{lem}

\begin{cor}\label{thm1} For $n\geq 1$ and $m\geq 1$, $R_t(C(2m-1,2n+1))$ is connected.
\end{cor}
\begin{proof}Note that $C(2m-1,2n+1)$ has $2m(2n+1)$ vertices and a perfect matching. Since $C(2m-1,2n+1)$ can be also viewed as a plane graph with two odd faces $f_1$ and $f_2$  where $f_2$ is the outer face (see Fig. \ref{D2}). Thus $R_t(C(2m-1,2n+1))$ is isomorphic to $R(C(2m-1,2n+1))$. By Theorem \ref{2.2.2}, it suffices to prove that $I[C]$ is elementary bipartite for each nice cycle $C$ of $C(2m-1,2n+1)$.
\begin{figure}[h]
\centering
\includegraphics[height=3.5cm,width=4.5cm]{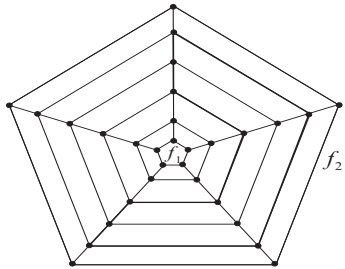}
\caption{\label{D2}The two odd faces $f_1$ and $f_2$ in $C(5,5)$.}
\end{figure}

Since the outer face of $I[C]$ is bounded by $C$, the outer face of $I[C]$ is even. Any interior face of $I[C]$ is either a square or odd face $f_1$. By face-degree sum formula of a plane graph, we obtain that $I[C]$ has even number of odd faces. Thus $I[C]$ has no odd faces and it is a polyomino. By Lemma \ref{2.1.1}, $I[C]$ is bipartite.  Since $C$ is a nice cycle of $I[C]$, $I[C]$ is elementary by Lemma \ref{2.1.3}.
\end{proof}

For $n\geq 1$ and $m\geq 1$, $R_t(C(2m-1,2n+1))$ and $\mathcal{T}(C(2m,2n+1))$ are isomorphic. By Corollary \ref{thm1}, we obtain the following result.
\begin{cor}For $n,m\geq 1$, $\mathcal{T}(C(2m,2n+1))$ is connected.
\end{cor}

\begin{remark}For $n,m\geq 1$, $\mathcal{T}(C(2m,2n))$ is not connected. For example, a tiling of quadriculated cylinder $C(4,6)$ in Fig. \ref{idu} is an isolated vertex of  $\mathcal{T}(C(4,6))$.
\begin{figure}[h]
\centering
\includegraphics[height=2.2cm,width=9.5cm]{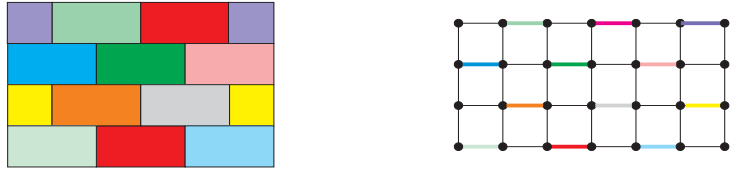}
\caption{\label{idu}A tiling of $C(4,6)$ which is an isolated vertex of $\mathcal{T}(C(4,6))$.}
\end{figure}
\end{remark}

\section{\normalsize The flip graph of quadriculated torus}
In this section, we will prove that $R_t(T(2n+1,2m))$ has exactly two components and such two components are isomorphic where $n\geq 1$ and $m\geq 2$.
Note that the flip graph and total resonance graph of $T(2n+1,2m)$ are isomorphic. As a corollary, we obtain corresponding result for flip graph of $T(2n+1,2m)$.

First we give some notations of quadriculated torus $T(2n+1,2m)$ for $n\geq 1$ and $m\geq 2$.
Since $T(2n+1,2m)$ is obtained from a $(2n+1)\times 2m$ chessboard by sticking left and right sides as well as top and bottom sides together, the vertices of $T(2n+1,2m)$ can be divided into $2n+1$ rows and $2m$ columns according to the arrangement of vertices in the chessboard. Denote the vertices of $T(2n+1,2m)$ by $(u_i,v_j)$ for $1\leq i \leq 2m$, $1\leq j\leq 2n+1$. Since the subgraph induced by the vertices of
$\{(u_i,v_j)|1\leq j\leq 2n+1\}$ is a cycle of length $2n+1$, we denote it by $C^i_{2n+1}$ for some $1\leq i\leq 2m$. Similarly, the subgraph induced by the vertices of
$\{(u_i,v_j)|1\leq i\leq 2m\}$ is denoted by $C^j_{2m}$ for some $1\leq j\leq 2n+1$.
\begin{figure}[h]
\centering
\includegraphics[height=4cm,width=7.5cm]{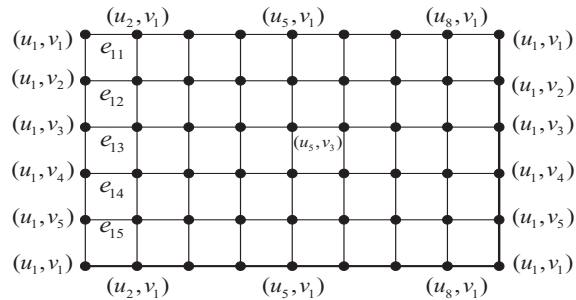}
\caption{\label{M42}Some notations of $T(5,8)$.}
\end{figure}
For $1\leq i\leq 2m$, let $E_i=\{e_{ij}=(u_i,v_j)(u_{i+1},v_j)|1\leq j\leq 2n+1\}$ where the subscript $i+1$ is considered modulo $2m$. Some notations of $T(5,8)$ are shown in Fig. \ref{M42}. In this drawing, each edge of $E_i$ is horizontal and we call it \emph{horizontal edge}. Define two perfect matchings of $T(2n+1,2m)$ as
 $$M_1=E_1\cup E_3\cup \cdots \cup E_{2m-1} \text{ and } M_2=E_2\cup E_4\cup \cdots\cup E_{2m}.$$

\begin{lem}\label{l1}$M_1$ and $M_2$ lie in different components of $R_t(T(2n+1,2m))$ for $n\geq 1$ and $m\geq 2$.
\end{lem}
\begin{proof}Suppose to the contrary that $M_1$ and $M_2$ lie in some same component of $R_t(T(2n+1,2m))$. Then there is a path in $R_t(T(2n+1,2m))$ connecting $M_1$ and $M_2$. That is, $M_2=M_1\oplus S_1\oplus\cdots \oplus S_k$ where $S_i$ is a facial cycle of $T(2n+1,2m)$ for $i=1,2,\dots,k$.

For $1\leq j\leq 2n+1$, since $e_{1j}\in M_1$ but $e_{1j}\notin M_2$, $e_{1j}$ appears odd times in the sequence $S_1,S_2,\dots,S_k$. Let $f_{j-1}$ and $f_j$ be the two facial cycles of $T(2n+1,2m)$ having edge $e_{1j}$ in common (the subscripts modulo $2n+1$). Then $\delta(f_{j-1})+\delta(f_{j})\equiv 1$ (mod 2). Adding up $2n+1$ equations together, we have $2[\delta(f_{1})+\delta(f_{2})+\cdots+\delta(f_{2n+1})]\equiv 1$ (mod 2), which is a contradiction.
\end{proof}

For a matching $M$, the vertices incident to the edges of $M$ are \emph{saturated} by $M$, the others are \emph{unsaturated}. For a vertex subset $X$ of $V(G)$, the set of edges with exactly one end vertex in $X$ are written $\nabla(X)$. For a subgraph $G'$ of $G$, we write $\nabla(G')$ instead of $\nabla(V(G'))$.
Lemma \ref{l1} implies that $R_t(T(2n+1,2m))$ has at least two components. In the sequel, we will prove that $R_t(T(2n+1,2m))$ has exactly two components. First of all, we give a significant lemma as follows.

\begin{lem}\label{l3}Let $M$ be a perfect matching of $T(2n+1,2m)$. Then we can obtain a perfect matching $M'$ from $M$ by a series of flips such that $M'\cap E_i=\emptyset$ for some $1\leq i\leq 2m$.
\end{lem}
\begin{proof}Let $l$ represent the number of horizontal edges in $M$. We will prove the lemma by induction on $l$.

For $1\leq j\leq 2m$, since $C^j_{2n+1}$ has $2n+1$ vertices, $|M\cap \nabla(C^j_{2n+1})|$ is odd. Thus we have $l\geq m$, and equality holds if and only if  $|M\cap \nabla(C^j_{2n+1})|=1$ for each $1\leq j\leq 2m$. For $l=m$, we have $|M\cap \nabla(C^2_{2n+1})|=1$. Then $M\cap E_1$ or $M\cap E_2$ is an empty set, and $M$ is the required perfect matching.

For $l\geq m+1$, we have $|M\cap \nabla(C^j_{2n+1})|\geq 3$ for some $1\leq j\leq 2m$. Since $|M\cap \nabla(C^j_{2n+1})|$ is odd and $M\cap \nabla(C^j_{2n+1})$ is contained in $E_{j-1}\cup E_j$ (the subscripts are modulo $2m$), there exists a path on $C^j_{2n+1}$ such that its two end vertices are saturated by $M\cap \nabla(C^j_{2n+1})$ and all intermediate vertices (if exists) are saturated by $M\cap E(C^j_{2n+1})$. Without loss of generality, we suppose that $P^j=(u_j,v_{1})(u_j,v_{2})\cdots(u_j,v_{2t+2})$ is such a path for some  $0\leq t\leq n-1$, and $(u_j,v_1)(u_{j+1},v_1)$ and $(u_j,v_{2t+2})(u_{j+1},v_{2t+2})$ are contained in $M\cap E_j$.
A \emph{ladder} is obtained by connecting each vertex of a path of length odd with the corresponding vertex on its copy with an edge. A ladder and its perfect matching shown in Fig. \ref{M2}(a) together form a \emph{ladder structure} where the boundary of the ladder form an $M$-alternating cycle. Next we prove that such ladder structure exists.

Let $Q^{j+1}=(u_{j+1},v_{2})(u_{j+1},v_{3})\cdots(u_{j+1},v_{1+2t})$ be a path of length $2t$ on $C^{j+1}_{2n+1}$.
Note that $|M\cap \nabla(Q^{j+1})|\geq 0$ is even.
If $|M\cap \nabla(Q^{j+1})|=0$ (see Fig. \ref{M2}(b)), then $\{(u_{j+1},v_{2})(u_{j+1},\\v_{3}),(u_{j+1},v_{4})(u_{j+1},v_{5}),\dots,(u_{j+1},v_{2t})(u_{j+1},v_{1+2t})\}\subseteq M$. Thus the subgraph induced by $V(P^j)\cup V(Q^{j+1})\cup \{(u_{j+1},v_1),(u_{j+1},v_{2+2t})\}$ form a ladder structure. Let $$T_x=(u_{j},v_{2x})(u_{j+1},v_{2x})(u_{j+1},v_{1+2x})(u_{j},v_{1+2x})(u_{j},v_{2x})\text{ for }1\leq x\leq t$$ be $t$ $M$-alternating facial cycles of $T(2n+1,2m)$, and let $$W_x=(u_{j},v_{1+2x})(u_{j+1},v_{1+2x})(u_{j+1},v_{2+2x})(u_{j},v_{2+2x})(u_{j},v_{1+2x})\text{ for }0\leq x\leq t$$ be $t+1$ $M\oplus T_1\oplus \cdots \oplus T_t$-alternating facial cycles of $T(2n+1,2m)$.
Let $ M''=M\oplus T_1\oplus \cdots \oplus T_t\oplus W_0\oplus \cdots \oplus W_t$ (see Fig. \ref{M2}(c)).
Then $M''$ is a perfect matching of $T(2n+1,2m)$ such that $l''=l-2$.

If $|M\cap \nabla(Q^{j+1})|>0$, then we assume that $(u_{j+1},v_{s_1})(u_{j+2},v_{s_1})$ and $(u_{j+1},v_{s_2})(u_{j+2},v_{s_2})$ are the two edges of $M\cap \nabla(Q^{j+1})$ such that $s_1$ and $s_2$ are subscripts of $v$ smaller than the other subscripts in edges of $M\cap \nabla(Q^{j+1})$ where $2\leq s_1,s_2\leq 2t+1$. Let $P^{j+1}$ be the sub-path of $Q^{j+1}$ with both end vertices being $(u_{j+1},v_{s_1})(u_{j+2},v_{s_1})$ and $(u_{j+1},v_{s_2})(u_{j+2},v_{s_2})$.
\begin{figure}[h]
\centering
\includegraphics[height=7.4cm,width=14.5cm]{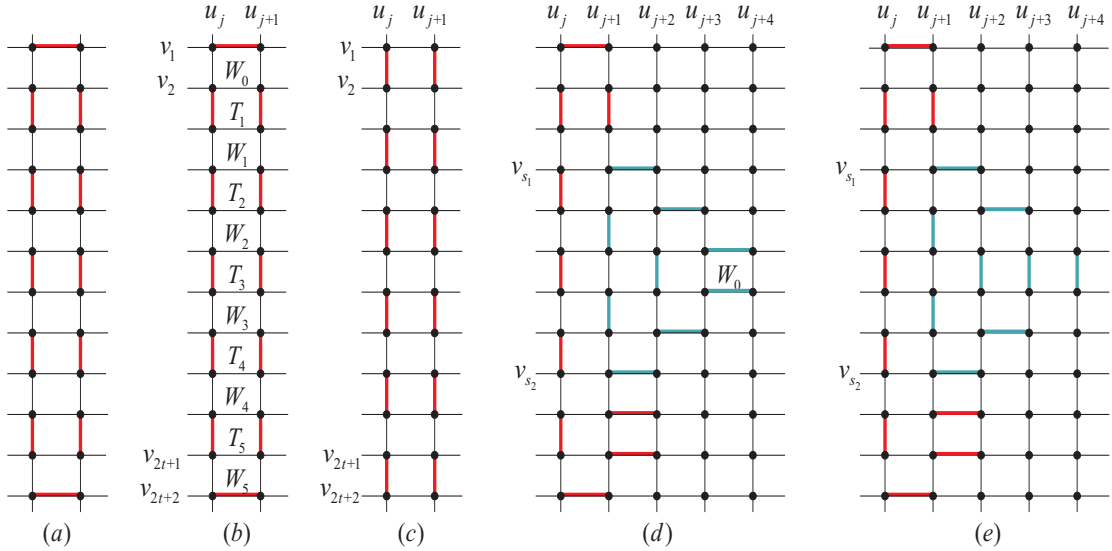}
\caption{\label{M2}Perfect matchings $M$ and $M''$ (shown as bold lines) in the proof of Lemma \ref{l3}.}
\end{figure}
Replacing $P^{j}$ with $P^{j+1}$, we repeat the above process. Since vertices on one row are finite, we will get some $Q^l$ such that $|M\cap \nabla(Q^{l})|=0$ shown in Fig. \ref{M2}(d). That is to say, there is a ladder structure. As we have proved above, we can obtain a perfect matching $M''$ shown in Fig. \ref{M2}(e) from $M$ by a series of flips such that $l''=l-2$.

By the induction hypothesis, we can obtain a perfect matching $M'$ from $M''$ by a series of flips such that $M'\cap E_i=\emptyset$ for some $1\leq i\leq 2m$. Since $M'$ is also obtained from $M$ by a series of flips, we obtain the required result.
\end{proof}

For an edge subset $F$ of $G$, we write $G-F$ for the subgraph of $G$ obtained by deleting set of edges $F$.
In the following, we will prove the main result.

\begin{thm}\label{t2}For $n\geq1$ and $m\geq 2$, $R_t(T(2n+1,2m))$ has exactly two components.
\end{thm}
\begin{proof}By Lemma \ref{l1}, $R_t(T(2n+1,2m))$ has at least two components containing $M_1$ and $M_2$ separately. It suffices to prove that $R_t(T(2n+1,2m))$ has a path from $M_1$ or $M_2$ to any other vertex.

Let $M$ be a perfect matching of $T(2n+1,2m)$ different from $M_1$ and $M_2$. By Lemma \ref{l3}, we can obtain a perfect matching $M'$ from $M$ by some flips such that $M'\cap E_i=\emptyset$ for some $1\leq i\leq 2m$. Let $G'=T(2n+1,2m)-E_i$. Then $G'$ is a graph obtained from a $(2n+1)\times (2m-1)$ chessboard by sticking top and bottom sides together. So $G'$ is a quadriculated cylinder $C(2m-1,2n+1)$.

Since $M'\cap E_i=\emptyset$, $M'$ is a perfect matching of $G'$. If $i$ is odd (resp. even), then $M_2$ (resp. $M_1$) is a perfect matching of $G'$. We assume that $M_j$ is a perfect matching of $G'$ for some $j\in\{1,2\}$. By Corollary \ref{thm1}, $R_t(G')$ is connected. So $R_t(G')$ has a path between $M'$ and $M_j$. Since $G'$ is quadriculated, $M_j$ is obtained from $M'$ by some flips. Combining that each face in $G'$ corresponds to a face of $T(2n+1,2m)$, we have that $M_j$ is obtained from $M'$ by some flips in $T(2n+1,2m)$. Since $M_j$ is also obtained from $M$ by some flips in $T(2n+1,2m)$, $R_t(T(2n+1,2m))$ has a path between $M_j$ and $M$. \end{proof}

Let $H_i$ be the component of $R_t(T(2n+1,2m))$ containing $M_i$ for $i=1,2$.
\begin{thm}\label{t3}For $n\geq1$ and $m\geq 2$, the two components $H_1$ and $H_2$ of $R_t(T(2n+1,2m))$ are isomorphic.
\end{thm}
\begin{proof}Define $\varphi((u_i,v_j))=(u_{i+1},v_j)$ for $1\leq i\leq 2m$ and $1\leq j\leq 2n+1$ where the subscript $i+1$ is considered modulo $2m$. Clearly, $\varphi$ is a translation of such torus and thus an automorphism on $T(2n+1,2m)$ which preserves the faces. Next we will prove that $\varphi$ induces an isomorphism from $H_1$ to $H_2$.

Since translating a perfect matching of $T(2n+1,2m)$ by length of one edge horizontally to the left or right, we still obtain a perfect matching of $T(2n+1,2m)$.
Thus, $\varphi$ induces a bijection on the set of perfect matchings. Since $\varphi(M_1)=M_2$, it suffices to prove that $\varphi$ preserves flip operation.

For $M^1M^2\in E(H_1)$, we have $M^2=M^1\oplus S$ for a facial cycle $S$ of $T(2n+1,2m)$. Since $\varphi(M_1-S)=\varphi(M_1)-\varphi(S)$, we have
\begin{eqnarray*}
\varphi(M^2)&=&\varphi((M^1-S)\cup(S-M_1))\\
            &=&\varphi(M^1-S)\cup \varphi(S-M^1)\\
            &=&(\varphi(M^1)-\varphi(S))\cup (\varphi(S)-\varphi(M_1))\\
            &=&\varphi(M_1)\oplus \varphi(S).
\end{eqnarray*}
Since $\varphi(S)$ is also a facial cycle of $T(2n+1,2m)$, we have $\varphi(M^1)\varphi(M^2)\in E(H_2)$.
For $\varphi(M^1)\varphi(M^2)\in E(H_2)$, we have $\varphi(M^2)=\varphi(M^1)\oplus S'$ for a facial cycle $S'$ of $T(2n+1,2m)$. Since $\varphi(M^2)=\varphi(M^1)\oplus S'=\varphi(M^1\oplus \varphi^{-1}(S'))$ by a similar proof as above, and $\varphi$ is an injective, we obtain that $M^2=M^1\oplus \varphi^{-1}(S')$. Since $\varphi^{-1}(S')$ is also a facial cycle of $T(2n+1,2m)$, we have $M^1M^2\in E(H_1)$.
\end{proof}
Two perfect matchings $M^1$ and $M^2$ of a graph $G$ are called \emph{equivalent} if there is an automorphism $\varphi$ of $G$ such that $\varphi{(M^1)}=M^2$. By proof of Theorem \ref{t3}, we obtain the following result.
\begin{cor}\label{4.5} For any perfect matching $M\in V(H_1)$, we have $\varphi(M)\in V(H_2)$, and $M$ and $\varphi(M)$ are equivalent.
\end{cor}

Since the inner dual of $T(2n+1,2m)$ is itself, $\mathcal{T}(T(2n+1,2m))$ is isomorphic to $R_t(T(2n+1,2m))$. By Theorems \ref{t2} and \ref{t3}, we obtain a  corollary as follows.
\begin{cor}For $n\geq 1$ and $m\geq 2$, $\mathcal{T}(T(2n+1,2m))$ has exactly two isomorphic components.
\end{cor}

A graph is  \emph{trivial} if it contains only one vertex; Otherwise, it is \emph{non-trivial}. It is more complicated for components of $R_t(T(2n,2m))$ for $n\geq 2$ and $m\geq 2$.
\begin{remark}$R_t(T(4,4))$ has 17 components, 12 of which are trivial.

{\rm Calculating by computer, we obtain that $T(4,4)$ has 272 perfect matchings. By analysing flip operations for each perfect matching, we get that $R_t(T(4,4))$ has twelve trivial components and five non-trivial components.}
\begin{figure}[h]
\centering
\includegraphics[height=3.4cm,width=3.6cm]{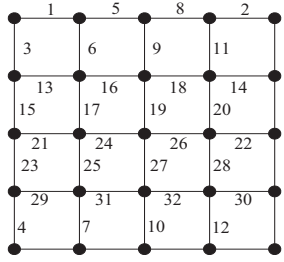}
\caption{\label{c4c4}Labels of edges of $T(4,4)$.}
\end{figure}

{\rm Label edges of $T(4,4)$ as in Fig. \ref{c4c4}. Then twelve trivial components are $M_{14}=\{1, 8, 14,\\ 16, 21, 26, 30, 31\}$, $M_{72}= \{2, 5, 13,18, 22, 24, 29, 32\}$, $M_{196}=\{3, 7, 9, 12, 17, 20, 23, 27\}$,  $M_{237}=\{4, 6, 10, 11, 15, 19, 25, 28\}$, $M_{49}=\{1, 9, 12, 14, 17, 23, 26, 31\}$, $M_{55}= \{1, 10, 11,15,\\ 16, 25, 26, 30\}$, $M_{116}=\{2, 6, 10, 15, 18, 24, 28, 29\}$,  $M_{120}=\{2, 7, 9, 13, 20, 23, 24, 32\}$, $M_{152}\\=\{3, 5, 12, 17, 18, 22, 27, 29\}$, $M_{187}= \{3, 7, 8, 16, 20, 21, 27, 30\}$, $M_{206}=\{4, 5, 11, 13, 19, 22,\\ 25, 32\}$,  $M_{223}=\{4, 6, 8, 14, 19, 21, 28, 31\}$. The
five non-trivial components contain respectively $M_1=\{1, 8, 13, 18, 21, 26, 29, 32\}$, $M_{2}=\{1, 8, 13, 18, 21, 26, 30, 31\}$,  $M_{69}=\{2, 5, 13,\\ 18, 21, 26, 29, 32\}$, $M_{178}=\{3, 6, 9, 12, 20, 23, 25, 27\}$, $M_{181}=\{3, 6, 10, 11, 19, 23, 25, 28\}$ where the component containing $M_1$ has 132 vertices and each other component has 32 vertices.}
\end{remark}

\section{\normalsize The forcing numbers of domino tilings}
For a quadriculated region, the forcing number of a tiling is equal to that of the corresponding perfect matching in its inner dual.
First, we give some concepts about the forcing number of a perfect matching, which was first introduced by Harary et al. \cite{2} and by Klein and Randi\'{c} \cite{3} in chemical literatures and which has important applications in resonance theory.
Let $M$ be a perfect matching of a graph $G$. A subset $S\subseteq M$ is called a \emph{forcing set} of $M$ if $S$ is not contained in any other perfect matching. The smallest cardinality of a forcing set of $M$ is called the \emph{forcing number} of $M$, denoted by $f(G,M)$.
 Afshani et al. \cite{5} obtained the following result.
\begin{lem}\cite{5} \label{4.3} If $M$ is a perfect matching of $G$ and $C$ is an $M$-alternating cycle of length 4, then $$|f(G,M\oplus E(C))-f(G,M)|\leq 1.$$
\end{lem}

The \emph{minimum} (resp. \emph{maximum}) \emph{forcing number} of $G$ is the minimum (resp. maximum) values of $f(G,M)$ over all perfect matchings $M$ of $G$, denoted by $f(G)$ and $F(G)$, respectively.
Afshani et al. \cite{5} showed that $F(C(2k-1,2n))=kn$ and $F(C(2k,2n))=kn+1$, and proposed an open problem: what is the maximum forcing number of non-bipartite graph $C(2m-1,2n+1)$? Jiang and Zhang \cite{29} solved the problem and obtained that $F(C(2m-1,2n+1))=(n+1)m$.

The \emph{forcing spectrum} of $G$ is the set of forcing numbers of all perfect matchings of $G$, which is denoted by $\text{Spec}(G)$. If $\text{Spec}(G)$ is an integer interval, then we say it is \emph{continuous}.
By proving that the resonance graph of a polyomino is connected, Zhang and Jiang \cite{41} obtained that its forcing spectrum is continuous.

In this section, we give some results about the minimum and maximum forcing numbers of  $C(2m-1,2n+1)$ and $T(2m,2n+1)$, and prove that their forcing spectra are continuous. Hence we obtain some corresponding results of forcing numbers of domino tilings in $2m\times (2n+1)$ quadriculated cylinder and torus.

\subsection{\normalsize The forcing numbers of domino tilings of $T(2n+1,2m)$}
In this subsection, we obtain the maximum forcing number of $T(2n+1,2m)$ by the method of Kleinerman \cite{16} and prove that its forcing spectrum is continuous. As a corollary, we obtain some results of forcing numbers of domino tilings in $T(2n+1,2m)$. For an edge subset $S$ of a graph $G$, we use $V(S)$ to denote the set of all end vertices of edges in $S$.

\begin{thm}\label{Max1} $F(T(2n+1,2m))=(n+1)m$.
\end{thm}
\begin{proof}
Let $G=T(2n+1,2m)$. First we prove that $F(G)\leq (n+1)m$.

For $1\leq i\leq 2m$, let $X_i=\{(u_i,v_j)|j=1,3,\dots, 2n-1\}$ and $Y_i=\{(u_i,v_j)|j=2,4,\dots, 2n\}$. Let $k=\lfloor\frac{m}{2}\rfloor$. If $m=2k$, then we take $$T=X_1\cup X_5\cup \cdots \cup X_{4(k-1)+1}\cup Y_3\cup Y_7\cup \cdots \cup Y_{4(k-1)+3}.$$
If $m=2k+1$, then we take $$T=X_1\cup X_5\cup \cdots \cup X_{4(k-1)+1}\cup Y_3\cup Y_7\cup \cdots \cup Y_{4(k-1)+3}\cup Y_{4k+1}.$$
So $T$ is of size $mn$ and we call $T$ \emph{marked vertices} shown in Fig. \ref{markedvertices}.
\begin{figure}[h]
\centering
\includegraphics[height=4.7cm,width=14cm]{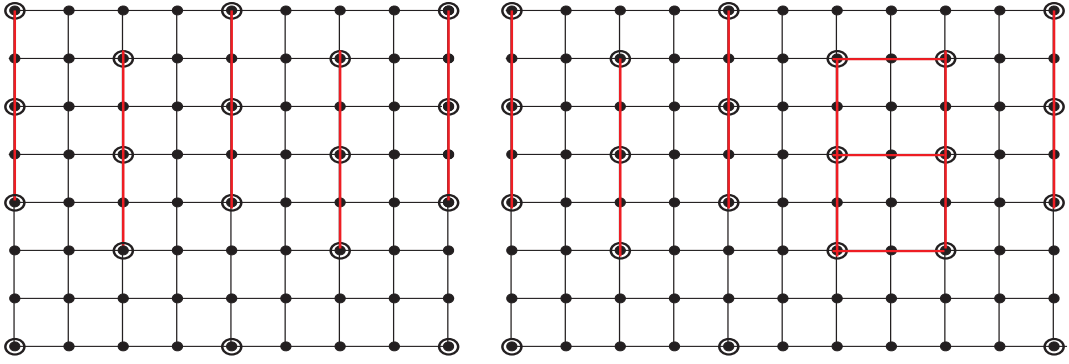}
\caption{\label{markedvertices}Marked vertices of $T(7,8)$ (left) and $T(7,10)$ (right).}
\end{figure}
For any perfect matching $M$ of $G$, we can produce a subset $E_T\subseteq M$ in which an edge lies in $E_T$ if one of its end vertices lies in $T$. Since $M$ is a matching and $T$ is an independent set, we have $|E_T|=|T|$.

We claim that $G[V(E_T)]$ contains no $M$-alternating cycles. Suppose to the contrary that $G[V(E_T)]$ contains an $M$-alternating cycle $C$. Since $T$ is an independent set, there is exactly half vertices of $C$ in $T$, and marked vertices and unmarked vertices appear alternately on $C$. So we can treat $C$ as a union of paths of length two whose two end vertices are marked vertices.
Let $G'$ be a subgraph of $G$ consisting of all paths of length two between marked vertices. Then each component of $G'$ is either a path with $2n-1$ vertices or a plane bipartite graph with $5n-2$ vertices shown in Fig. \ref{markedvertices} as bold lines.
Since $C$ is a cycle of $G'$, $m$ is odd and $C$ is contained in the component of plane bipartite graph. For each cycle $C'$ of $G'$, $C'$ is also a cycle of $G$. Since there is always an odd number of vertices in $G$ sequestered inside $C'$, $C'$ cannot be $M$-alternating, a contradiction.
By the claim, $G[V(E_T)]$ has a unique perfect matching, and $M\setminus E_T$ is a forcing set of $M$.
Since $f(G,M)\leq |M\setminus E_T|=(n+1)m,$ we have $F(G)\leq (n+1)m$ by the arbitrariness of $M$.

It suffices to prove that $f(G,M_1)\geq (n+1)m$ where $M_1=E_1\cup E_3\cup \cdots \cup E_{2m-1}$.
For $1\leq i\leq m$, let $G^i=G[V(C^{2i-1}_{2n+1})\cup V(C^{2i}_{2n+1})]$ and $M^i=M\cap E(G^i)=E_{2i-1}$. Then $M^i$ is a perfect matching of $G^i$, and it is clear that $f(G^i,M^i)=n+1$. Thus, we obtain that $f(G,M_1)\geq\sum_{i=1}^mf(G^i,M^i)=(n+1)m.$
\end{proof}

\begin{cor}\label{c2}For $n\geq 1$ and $m\geq 2$, $Spec(T(2n+1,2m))$ is continuous.
\end{cor}
\begin{proof}Since any two perfect matchings of the same component of $R_t(T(2n+1,2m))$ can be obtained from each other by some flips, by Lemma \ref{4.3}, the forcing numbers of  perfect matchings in the same component form an integer interval. By Theorem \ref{t2}, $R_t(T(2n+1,2m))$ has exactly two components containing $M_1$ and $M_2$ separately. Since both $f(T(2n+1,2m),M_1)$ and $f(T(2n+1,2m),M_2)$ equal the maximum forcing number $(n+1)m$, $\text{Spec}(T(2n+1,2m))$ is an integer interval.
\end{proof}

By Theorem \ref{Max1} and Corollary \ref{c2}, we obtain the following result.

\begin{cor} For $n\geq 1$ and $m\geq 2$, the forcing numbers of all tilings of $T(2n+1,2m)$ form an integer interval whose maximum value is $(n+1)m$.
\end{cor}

Riddle \cite{7} and Kleinerman \cite{16} obtained separately that $f(T(2m,2n))=2\text{min}\{m,n\}$ and $F(T(2m,2n))=mn$. But the forcing spectrum of $T(2m,2n)$ is not necessarily continuous. For example, we obtain that $\text{Spec}(T(4,10))=\{4,6,7,8,9,10\}$ by calculating forcing numbers of all 537636 perfect matchings by computer, which has one gap 5.

\subsection{\normalsize The forcing numbers of domino tilings of  $C(2m,2n+1)$}
In this subsection, we prove that the forcing spectrum of $C(2m-1,2n+1)$ is continuous. As a corollary, we obtain that the forcing numbers of all domino tilings of $C(2m,2n+1)$ form an integer interval. Last we give some results about the minimum forcing numbers of special quadriculated cylinders.

\begin{cor}\label{2.2.6} For $n\geq 1$ and $m\geq 1$, $\text{Spec}(C(2m-1,2n+1))$ is continuous.
\end{cor}
\begin{proof}By Corollary \ref{thm1}, $R_t(C(2m-1,2n+1))$ is connected. Then any two perfect matchings of $C(2m-1,2n+1)$ can be obtained from each other by a series of flips.
By Lemma \ref{4.3}, a flip cannot change the forcing number by more than one. So the forcing spectrum of $C(2m-1,2n+1)$ form an integer interval.
\end{proof}

Recall that $F(C(2m-1,2n+1))=(n+1)m$ \cite{29}. Combining Corollary \ref{2.2.6}, we obtain the following result.
\begin{cor} For $n\geq 1$ and $m\geq 1$, the forcing numbers of all tilings of  $C(2m,2n+1)$ form an integer interval whose maximum value is $(n+1)m$.
\end{cor}

\begin{remark} For $n\geq 1$ and $m\geq 1$, $\text{Spec}(C(2m-1,2n))$ is not necessarily continuous.
For example, we obtain that $\text{Spec}(C(1,10))=\{2,4,5\}$ and $\text{Spec}(C(1,12))=\{2,4,5,6\}$ by computer, both of which have one gap 3.
\end{remark}

For forcing numbers of perfect matchings of $C(2m-1,2n+1)$, there is a remaining problem as follows.
\begin{pb}What is $f(C(2m-1,2n+1))$ for $n\geq 1$ and $m\geq 1$?
\end{pb}

The minimum forcing numbers of some special quadriculated cylinders have been given (see the detail in Ref. \cite{Li} and we omit the proof here).

\begin{itemize}
\item $f(C(1,5))=2$.
\item $f(C(1,2n+1))=\lceil\frac{2n+1}{3}\rceil$ for $n\geq 1$.
\item $f(C(3,2n+1))=\lceil\frac{4n+2}{5}\rceil$  for $n\geq 1$.
\item $f(C(2m-1,3))=m-\lfloor\frac{m}{3}\rfloor$ for $m\geq 1$.
\item $f(C(2m-1,5))=m$ for $m\geq 2$.
\end{itemize}

\noindent{\normalsize \textbf{Acknowledgments}}
\smallskip

This work was supported by National Natural Science Foundation of China (Grant No. 11871256). We thank Yaxian Zhang providing a simpler method for the proof of Lemma \ref{l3} during the discussion.


\begin{thebibliography}{99}
\small \setlength{\itemsep}{-.2mm}
\bibitem{5}P. Afshani, H. Hatami, E. S. Mahmoodian, On the spectrum of the forced matching number of graphs, Australas. J. Combin. 30 (2004) 147-160.
\bibitem{T2}M. Aizenman, M. L. Valc\'{a}zar, S. Warzel, Pfaffian correlation functions of planar dimer covers, J. Stat. Phys. 166(3-4) (2017) 1078-1091.
\bibitem{E}N. Elkies, G. Kuperberg, M. Larsen, J. Propp, Alternating-sign matrices and domino tilings \uppercase\expandafter{\romannumeral1}, J. Algebraic Combin. 1(2) (1992) 111-132.
\bibitem{LW17}J. Freire, C. Klivans, P. H. Milet, N. C. Saldanha, On the connectivity of spaces of three-dimensional domino tilings, Trans. Amer. Math. Soc. 375(3) (2022) 1579-1605.
  \bibitem{G}W. Gr\"{u}ndler, Signifikante Elektronenstrukturen fur benzenoide Kohlenwasserstoffe, Wiss. Z. Univ. Halle 31 (1982) 97-116.

\bibitem{2}F. Harary, D. J. Klein, T. P. \v{Z}ivkovi\'{c}, Graphical properties of polyhexes: perfect matching vector and forcing, J. Math. Chem. 6(3) (1991) 295-306.
\bibitem{29}X. Jiang, H. Zhang, The maximum forcing number of cylindrical grid, toroidal 4-8 lattice and Klein bottle 4-8 lattice, J. Math. Chem. 54(1) (2016) 18-32.
\bibitem{J}P. John, H. Sachs, H. Zernitz, Counting perfect matchings in polyominoes with an application to the dimer problem, Zastosowania Matematyki Applicationes Mathematicae 19(3-4) (1987) 465-477.
\bibitem{Ka} P. W. Kasteleyn, The statistics of dimers on a lattice \uppercase\expandafter{\romannumeral1}. The number of dimer arrangements on a quadratic lattice, Physica 27(12) (1961) 1209-1225.
\bibitem{K}H. Kim,  S. Lee, S. Oh, Domino tilings for augmented Aztec rectangles and their chains, Electron. J. Combin. 26(3) (2019) 1-16.
\bibitem{3}K. J. Klein, M. Randi\'{c}, Innate degree of freedom of a graph, J. Comput. Chem. 8(4) (1987) 516-521.
\bibitem{16}S. Kleinerman, Bounds on the forcing numbers of bipartite graphs, Discrete Math. 306(1) (2006) 66-73.
\bibitem{LZ}P. C. B. Lam, H. Zhang, A distributive lattice on the set of perfect matchings of a plane bipartite
graph, Order 20(1) (2003) 13-29.
\bibitem{Li}C. Li, Forcing matching number and spectrum of graph $P_{2n}\times C_{2m+1}$, (in Chinese with an English summary), Master thesis, Lanzhou Univ., 2011.
\bibitem{14}L. Lov\'{a}sz, M. D. Plummer, Matching Theory, North-Holland, Amsterdam, 1986.
\bibitem{Pr93}J. G. Propp, Lattice structure for orientations of graphs, preprint, 1993.
\bibitem{RE}E. R\'{e}mila, The lattice structure of the set of domino tilings of a polygon, Theoret. Comput. Sci. 322(2) (2004) 409-422.

\bibitem{7}M. E. Riddle, The minimum forcing number for the torus and hypercube, Discrete Math. 245(1-3) (2002) 283-292.

\bibitem{S}N. C. Saldanha, Domino tilings of cylinders: connected components under flips and normal distribution of the twist, Electron. J. Combin. 28(1) (2021) 1-23.
\bibitem{O}N. C. Saldanha, C. Tomei, An overview of domino and lozenge tilings, Resenhas IME-USP, 2(2) (1995) 239-252.
\bibitem{T1}N. C. Saldanha, C. Tomei, M. A. Jr. Casarin, D. Romualdo, Spaces of domino tilings, Discrete Comput. Geom. 14(2) (1995) 207-233.
\bibitem{TF}H. N. V. Temperley, M. E. Fisher, Dimer problem in statistical mechanics--an exact result, Phil. Mag. 6(8) (1961), 1061-1063.
\bibitem{Th90}W. P. Thurston, Conway's tiling groups, Amer. Math. Monthly 97(8) (1990) 757-773.

\bibitem{TY}N. Tratnik, D. Ye, Resonance graphs and perfect matchings of graphs on surfaces, preprint,  2017,
    https://arxiv.org/pdf/1710.00761.pdf.
\bibitem{Z5}D. B. West, Introduction to Graph Theory, Prentice Hall, 2001, P. 239.
\bibitem{Z4}H. Zhang, The connectivity of Z-transformation graphs of perfect matchings of polyominoes, Discrete Math. 158(1-3) (1996) 257-272.
\bibitem{Zh06}H. Zhang, Z-transformation graphs of perfect matchings of plane bipartite
graphs: a Survey, MATCH Commun. Math. Comput. Chem. 56(3) (2006) 457-476.
\bibitem{Z1}F. Zhang, X. Guo, R. Chen, Z-transformation graphs of perfect matchings of hexagonal systems, Discrete Math. 72(1-3) (1988) 405-415.
 \bibitem{ZLS08}H. Zhang, P. C. B. Lam, W. C. Shiu, Resonance graphs and a binary coding for the 1-factors of benzenoid systems, SIAM J. Discrete Math. 22(3) (2008) 971-984.
\bibitem{41}H. Zhang, X. Jiang, Continuous forcing spectra of even polygonal chains, Acta Math. Appl. Sinica (English Ser.) 37(2) (2021) 337-347.
\bibitem{Z3}H. Zhang, F. Zhang, Plane elementary bipartite graphs, Discrete Appl. Math. 105(1-3) (2000) 291-311.
\bibitem{Z6}H. Zhang, F. Zhang, H. Yao, Z-transformation graphs of perfect matchings of plane bipartite graphs, Discrete Math. 276(1-3) (2004) 393-404.
\end{thebibliography}
\end{document}